\documentclass[12pt]{amsart}

\usepackage{amsfonts, amsmath, amssymb, amsthm, amscd, hyperref}

\usepackage{anysize}
\usepackage[all,cmtip]{xy}

\newtheorem{theorem}{Theorem}[section]
\newtheorem*{theorem*}{Theorem}

\newtheorem{corollary}[theorem]{Corollary}

\theoremstyle{definition}
\newtheorem{definition}[theorem]{Definition}

\theoremstyle{remark}
\newtheorem{remark}[theorem]{Remark}

\numberwithin{equation}{section}

\newcommand{\pt}{\mathrm{pt}}

\newcommand{\Sg}{\Sigma}

\DeclareMathOperator{\Hom}{Hom}

\renewcommand{\epsilon}{\varepsilon}

\newcommand{\barany}{B\'{a}r\'{a}ny}
\newcommand{\blagojevic}{Blagojevi\'{c}}

\newcommand{\dolnikov}{Dol'nikov}
\newcommand{\matousek}{Matou\v{s}ek}

\newcommand{\szucs}{Sz{\H{u}}cs}
\newcommand{\vrecica}{Vre\'{c}ica}
\newcommand{\vucic}{Vu\v{c}i\'{c}}
\newcommand{\zivaljevic}{\v{Z}ivaljevi\'{c}}

\newcommand{\RR}{\mathbb{R}} 
\newcommand{\FF}{\mathbb{F}} 
\newcommand{\RP}{\mathbb{R}P} 
\newcommand{\toG}[1]{\longrightarrow_{#1}} 

\DeclareMathOperator{\homeo}{\cong} 
\DeclareMathOperator{\iso}{\cong} 

\newcommand{\ceil}[1]{\left\lceil{#1}\right\rceil} 
\newcommand{\incl}{\hookrightarrow} 

\newcommand{\eps}{\varepsilon}

\DeclareMathOperator{\dotcup}{\dot{\cup}}

\begin{document}

\title{Projective center point and Tverberg theorems}

\author{Roman~Karasev$^{*}$}
\address{Roman Karasev, Dept. of Mathematics, Moscow Institute of Physics and Technology, Institutskiy per. 9, Dolgoprudny, Russia 141700; and Laboratory of Discrete and Computational Geometry, Yaroslavl' State University, Sovetskaya st. 14, Yaroslavl', Russia 150000}
\thanks{{$^{*}$} Supported by the Dynasty Foundation, the President's of Russian Federation grant MD-352.2012.1, the Russian Foundation for Basic Research grants 10-01-00096 and 10-01-00139, the Federal Program ``Scientific and scientific-pedagogical staff of innovative Russia'' 2009--2013, and the Russian government project 11.G34.31.0053.}
\email{r\_n\_karasev@mail.ru}

\author{Benjamin~Matschke$^{**}$}
\address{Benjamin Matschke, Institute for Advanced Study, Princeton}
\email{matschke@ias.edu}
\thanks{{$^{**}$} Supported by NSF Grant DMS-0635607.}

\date{July 10, 2012}


\begin{abstract}

We present projective versions of the center point theorem and Tverberg's theorem, interpolating between the original and the so-called ``dual'' center point and Tverberg theorems.

Furthermore we give a common generalization of these and many other known (trans\-ver\-sal, constraint, dual, and colorful) Tverberg type results in a single theorem, as well as some essentially new results about partitioning measures in projective space.

\end{abstract}

\maketitle

\section{Introduction}

In this paper we focus on two classical topics in discrete geometry:
the center point theorem from Neumann~\cite{neumann1945} and Rado~\cite{rado1947} (see also Gr\"unbaum~\cite{grun1960}) and Tverberg's theorem~\cite{tver1966}.
Many deep generalizations of these classical results have been made in the last three decades, starting from the topological generalization by \barany, Shlosman, and \szucs~\cite{bss1981}.
A good review on this topic and numerous references are given in \matousek's book~\cite{mat2003}.
After this book was published, new achievements were made by Hell~\cite{hell2007, hell2008}, Engstr\"om~\cite{eng2011} and Engstr\"om--Nor\'en~\cite{engnor2011}, K.~\cite{kar2008dcpt}, and \blagojevic--M.--Ziegler~\cite{bmz2009ct,bmz2011ctv}, establishing ``constrained'', ``dual'', and ``optimal colorful'' Tverberg type theorems.

The discrete center point theorem states the following:
For any finite set $X\subset \mathbb R^d$ there exists a \emph{center point} $c\in\mathbb R^d$ such that any closed half-space $H\ni c$ contains at least $\left\lceil\frac{|X|}{d+1}\right\rceil$ points of $X$.
In K.~\cite{kar2008dcpt} a dual center point theorem and a dual Tverberg theorem for families of hyperplanes were proved.
The dual center point theorem states:
For any family of $n$ hyperplanes in general position in $\mathbb R^d$ there exists a point $c$ such that any ray starting at $c$ intersects at least $\left\lceil\frac{n}{d+1}\right\rceil$ hyperplanes.

Here the use of the adjective ``dual'' is rather descriptive, it does not refer to projective duality.
Thus it is interesting to dualize it once more projectively and compare it with the original center point theorem.

\begin{definition}
\label{defPartitioningRPd}
Any two distinct hyperplanes $H_1$ and $H_2$ partition $\mathbb RP^d$ into two pieces. In this paper, we always consider the pieces as being closed. If $H_1$ and $H_2$ coincide then we consider $H_1=H_2$ as one piece and the whole $\mathbb RP^d$ as the other.
\end{definition}

The projective dual of the ``dual center point theorem'' reads:
Assume that $X$ is a family of $n$ points in $\mathbb RP^d$ and $c\in \mathbb RP^d$ is another point such that the family $X\cup c$ is in general position. Then there exists a hyperplane $W\subseteq \mathbb RP^d$ such that together with any hyperplane $H_1\ni c$ it partitions $\mathbb RP^d$ into two parts each containing at least $\left\lceil\frac{n}{d+1}\right\rceil$ points of~$X$.
From the proof of this theorem we can assure that $W$ does not contain~$c$; however if we omit the general position assumption then the theorem remains true by a compactness argument but $W$ may happen to contain~$c$.

Now we are going to interpolate between the original center point theorem and the latter ``dual to dual'' version (they appear as special cases when $V$ is the hyperplane at infinity or when $V$ is a point):

\begin{theorem}[Projective center point theorem]
\label{projective-cpt}
Suppose that $V\subset \mathbb RP^d$ is a projective subspace of dimension $v$ and $X$ is a finite point set with $|X|=n$. Put
$$
r = \left\lceil\frac{n}{(d-v)(v+1)+1}\right\rceil.
$$
Then there exists a projective subspace $W\subset \mathbb RP^d$ of dimension $d-v-1$ such that any pair of hyperplanes $H_1\supseteq V$ and $H_2\supseteq W$ partitions $\mathbb RP^d$ into two parts each containing at least $r$ points of $X$.

If we require the general position assumption, that no $r$ points of $X$ together with $V$ are contained in a hyperplane, then $W$ may be chosen disjoint from $V$.
\end{theorem}

The ordinary center point theorem~\cite{neumann1945,rado1947} is usually stated for measures, which follows from the discrete version by an approximation argument.
Here is the corresponding version:

\begin{theorem}[Projective center point theorem for measures]
\label{projective-cpt-mes}
Suppose that $V\subset \mathbb RP^d$ is a projective subspace of dimension $v$ and $\mu$ is a probability measure on $\mathbb RP^d$. Then there exists a projective subspace $W\subset \mathbb RP^d$ of dimension $d-v-1$ such that any pair of hyperplanes $H_1\supseteq V$ and $H_2\supseteq W$ partitions $\mathbb RP^d$ into two parts $P_1$ and $P_2$ so that
$$
\mu(P_1), \mu(P_2) \ge \frac{1}{(d-v)(v+1)+1}.
$$
\end{theorem}

\begin{remark}
If $\mu$ is absolutely continuous then the case $W\cap V\neq \emptyset$ is automatically excluded, because $H_1$ and $H_2$ must not coincide in this case.
\end{remark}

The rest of the paper is organized as follows:
We start with the proof of Theorem~\ref{projective-cpt} in Section \ref{cpr-proof-sec}.
In Section \ref{secTverbergInterpolationWithDual} we analogously interpolate between Tverberg's theorem and its dual.
In Section~\ref{trans-sec} we state a very general theorem incorporating almost all that we know about (dual, transversal, constrained, colorful) Tverberg type theorems.
In particular this implies a projective center transversal theorem, which generalizes Theorem \ref{projective-cpt-mes}, see Section \ref{secProjCenterTransversalThm}.
Exchanging quantors can possibly lead to other interesting theorems.
As an instance of this we prove another projective Tverberg theorem and a transversal generalization in Sections \ref{secAnotherProjTver} and \ref{secTranversalGeneralizationOfTverberg}.

\vspace*{3mm}
\noindent {\bf Open problems.}
Some of our theorems need technical assumptions in order to assure that certain topological obstructions do not vanish and such that we are able to prove that. These include the usual prime power assumption in Tverberg type theorems, as well as more specific assumptions in the results of Sections~\ref{secAnotherProjTver} and \ref{secTranversalGeneralizationOfTverberg}. It would be very interesting to know in which cases these assumptions are necessary and in which cases they can be avoided.

Moreover, Sections \ref{secAnotherProjTver} and \ref{secTranversalGeneralizationOfTverberg} already indicate that interesting versions and generalizations of our projective center point and Tverberg theorems might be obtained by exchanging quantors and by including additional subspaces on top of $V$ and $W$.
Is there a reasonable natural generalization incorporating all results of this paper? 

\vspace*{3mm}
\noindent {\bf Acknowledgements.}
We thank Karim Adiprasito, Pavle \blagojevic, Moritz Firsching, Peter Landweber, Louis Theran, and G\"{u}nter Ziegler for discussions.

\section{Proof of the projective center point theorem} \label{cpr-proof-sec}

In order to prove Theorem~\ref{projective-cpt} we first dualize it projectively (and interchange $v$ and $d-v-1$):
First note that dual to Definition \ref{defPartitioningRPd}, any two points $x,y\in\RP^d$ partition the space of hyperplanes in $\RP^d$ into two pieces, which are characterized by where a hyperplane intersects the (or a) line $\ell$ through $x$ and $y$, since $\ell$ gets partitioned into two closed parts by $x$ and $y$ as well.

\begin{theorem}
Suppose that $V\subset \mathbb RP^d$ is a projective subspace of dimension $v$ and $\Xi$ is a family of hyperplanes with $|\Xi|=n$. Put
$$
r = \left\lceil\frac{n}{(d-v)(v+1)+1}\right\rceil.
$$
Then there exists a projective subspace $W\subset \mathbb RP^d$ of dimension $d-v-1$ such that the following condition holds: For a pair of points $x\in V$ and $y\in W$ and any line $\ell\supseteq\{x,y\}$ the pair $(x,y)$ partitions the line $\ell$ into two parts and any of these parts intersects at least $r$ members of $\Xi$.
\end{theorem}

The general position assumption now becomes: No $r$ hyperplanes of $\Xi$ have a common point in $V$. In the following proof it will guarantee that $V$ and $W$ have no common point and the line $\ell$ is uniquely determined by the pair of points $x$ and $y$.

Let us go to the homogeneous coordinate space $\mathbb R^{d+1}$ with the corresponding subspace $\widehat V$ of dimension $v+1$ and a family of hyperplanes $\widehat \Xi$ given by homogeneous linear equations $\lambda_i(x)=0$ for $i=1,\ldots, n$. We have to find an appropriate subspace $\widehat W$ of dimension $d-v$.

Consider the formal simplex $\Delta^{n-1}$ with vertices indexed by $\{1, \ldots, n\}$. A point $p\in \Delta^{n-1}$ is an $n$-tuple of nonnegative reals $p=(t_1, \ldots, t_n)$ that sum to 1.
Choose an $\eps>0$ and consider the following quadratic form on $\mathbb R^{d+1}$:
$$
q_{p,\eps}(x) = \sum_{i=1}^n t_i \lambda_i(x)^2 + \eps\langle x,x\rangle,
$$
where $\langle x,x\rangle$ is the standard quadratic form on $\mathbb R^{d+1}$.
We would prefer to put $\eps=0$ but have to consider positive $\eps$ in order to make sure that $q_{p,\eps}(x)$ is positive definite for any $p\in\Delta^{n-1}$.

Let $\widehat W_{p,\eps}\subset\RR^{d+1}$ be the subspace that is orthogonal to $\widehat V$ with respect to the scalar product defined by $q_{p,\eps}(x)$.
Obviously $\widehat W_{p,\eps}$ depends continuously on $p\in \Delta^{n-1}$ and $\eps>0$ and $\widehat W_{p,\eps}$ always has trivial intersection with $\widehat V$. Note that the space of all possible linear subspaces $\widehat W$ of dimension $d-v$ having trivial intersection with $\widehat V$ is a Schubert cell in the Grassmannian $G(d+1, d-v)$, which is homeomorphic to $\RR^D$, $D=(d-v)(v+1)$.
Thus for a fixed $\eps>0$ we get a map $\Delta^{n-1}\to\RR^D$ that sends $p$ to $\widehat W_{p,\eps}$.

Now we invoke the topological center point theorem from K.~\cite{kar2010cpt}, which states that any continuous map from $\Delta^{n-1}$ to a metric space of covering dimension $D$, such as $\RR^D$, has the property that the intersection of the images of all faces of codimension $\left\lceil\frac{n}{D+1}\right\rceil-1$ in $\Delta^{n-1}$ is non-empty.

Thus in our situation we find a subspace $\widehat W_\eps$ with the following property: For any face $F\in \Delta^{n-1}$ of dimension at least $n-r$ (i.e. spanned by $n-r+1$ vertices of $\Delta^{n-1}$) there exists $p_{F,\eps}\in F$ such that $\widehat W_\eps=\widehat W_{p_{F,\eps},\eps}$.

By a compactness argument we can choose a sequence $\eps_m\to +0$ and corresponding points $p_{F,\eps_m}$ such that the respective $\widehat W_{\eps_m}$ converge to some $\widehat W$ in the Grassmannian topology, the points $p_{F,\eps_m}$ tend to some points $p_F$, and so do the quadratic forms $q_{p_{F,\eps_m},\eps_m}\to q_F$.

Consider a pair of lines $\widehat x\subseteq \widehat V$ and $\widehat y\subseteq \widehat W$ (all lines and subspaces in $\mathbb R^{d+1}$ are linear and pass through the origin). Let $\widehat\ell$ be a two-dimensional plane containing them both. The lines $\widehat x$ and $\widehat y$ break $\widehat\ell$ projectively in two closed parts: $P_1$ and $P_2$. We have to show that any of them intersects at least $r$ members of $\widehat \Xi$ nontrivially. Assume the contrary: a subfamily of $n-r+1$ hyperplanes (denote it $F$) intersect $P_1$ only at the origin; this means that every corresponding $\lambda_i$ is nonzero on $\widehat\ell$ and its zero set (minus the origin) is contained in the interior of $P_2$.

Consider now $F$ as a face of $\Delta^{n-1}$ and take its corresponding $p_F$ and $q_F$. The lines $\widehat x$ and $\widehat y$ must be orthogonal with respect to $q_F$ because the relation ``orthogonal'' depends continuously on everything. The quadratic form $q_F$ may be degenerate but it is nonzero on $\widehat\ell$, because all the corresponding hyperplanes intersect $\widehat\ell$ transversally. Consider two cases:

\begin{itemize}
\item Case~1. The lines $\widehat x$ and $\widehat y$ are different. Let them be the $x$ and the $y$ axes and let the interior of $P_2$ be defined by $xy<0$.
Then $q_F|_{\widehat\ell}$ is a convex combination of the forms $(a_ix+b_iy)^2$ with positive $a_i$ and $b_i$, $i\in F$. The corresponding scalar product between $(1,0)$ and $(0,1)$ is therefore $\sum_{i\in F} t_i a_ib_i>0$, which is a contradiction.
\item Case~2. The lines $\widehat x$ and $\widehat y$ coincide.
Since $\widehat x\perp_{q_F}\widehat y$, every $\lambda_i$ whose square comes with a positive coefficient $t_i$ in $q_F$ must be zero on $\widehat x=\widehat y$, and thus the corresponding hyperplane in $\widehat\Xi$ intersects both parts $P_1=\widehat x=\widehat y$ and $P_2=\widehat\ell$ nontrivially.
\end{itemize}

It remains to verify that the general position assumption implies $\widehat V\cap\widehat W = 0$.
Assuming the contrary, there are lines $\widehat x\subseteq\widehat V$ and $\widehat y\subseteq\widehat W$ with $\widehat x=\widehat y$.
For any $(n-r)$-face $F$ of $\Delta^{n-1}$ the above Case~2 implies that at least one hyperplane of $\widehat\Xi$ corresponding to a vertex of $F$ contains $\widehat x=\widehat y$.
Hence at least $r$ hyperplanes of $\widehat\Xi$ contain $\widehat x\subseteq \widehat V$ and this is exactly what we exclude by the general position assumption.
\qed

\section{A projective Tverberg theorem} \label{secTverbergInterpolationWithDual}

Instead of the topological center point theorem from~\cite{kar2010cpt} we could invoke the topological Tverberg theorem from~\cite{bss1981,oza1987,vol1996} (see also~\cite{mat2003}) to obtain the following:

\begin{theorem}
\label{projective-tver}
Suppose that $V\subset \mathbb RP^d$ is a projective subspace of dimension $v$ and $X$ is a finite point set with $|X|=(D+1)(r-1)+1$, where $D=(d-v)(v+1)$ and $r$ is a prime power.

Then there exists a projective subspace $W\subset \mathbb RP^d$ of dimension $d-v-1$ and a partition of $X$ into $r$ subfamilies $X_j$ $(j=1,\ldots,r)$ such that any pair of hyperplanes $H_1\supseteq V$ and $H_2\supseteq W$ partitions $\mathbb RP^d$ into two parts so that each part has nonempty intersection with every $X_j$.
\end{theorem}

\begin{remark}
An additional general position assumption ``no $r$ points of $X$ together with $V$ are contained in a hyperplane'' will again guarantee that $W$ may be chosen disjoint from $V$.
\end{remark}

\begin{remark}
The number of partitions of $X$ into $r$ subfamilies such that there exist a $W$ as in the theorem can be bounded from below as in \vucic--\zivaljevic~\cite{vucziv1993} (if $r$ is a prime) and in Hell~\cite{hell2007} (for $r=p^\ell$) by $\frac{1}{(r-1)!}\left(\frac{r}{\ell+1}\right)^{\ceil{(r-1)(d+1)/2}}$.
\end{remark}

\begin{remark}
There is also a colorful extension as in \blagojevic--M.--Ziegler~\cite{bmz2009ct} (compare with Remark~\ref{remColorfulProjectiveTransversalTverberg}) and ``constraint Tverberg''-extensions as in Hell \cite{hell2008} and Engstr\"om--Nor\'en \cite{engnor2011}.
\end{remark}

\begin{proof}
We start as in Section~\ref{cpr-proof-sec} by dualizing this statement and considering the continuous map from $\Delta^{n-1}$ ($n=(D+1)(r-1)+1$) to subspaces $\widehat W_{p,\eps}$. By the topological Tverberg theorem~\cite{bss1981,oza1987,vol1996} there exist $r$ points $p_1,\ldots, p_r\in \Delta^{n-1}$ with disjoint supports (corresponding to subfamilies $\Xi_j$) such that their corresponding subspaces $\widehat W_{p_j,\eps}$ coincide.

After that we pass to the limit $\eps_m\to 0$ and repeat the same arguments.
The degenerate case $\widehat V\cap \widehat W\neq 0$ is possible only when every $\Xi_j$ has a hyperplane containing the same common line $\widehat x\in \widehat V\cap \widehat W$. So again it is excluded by the general position assumption.
\end{proof}

\section{A transversal projective Tverberg theorem} \label{trans-sec}

Tverberg and \vrecica~\cite{tv1993grthst} conjectured a common generalization of the ham sandwich theorem, the center transversal theorem, and Tverberg's theorem, which they proved in some special cases. Later \zivaljevic~\cite{ziv1999}, \vrecica~\cite{vre2003}, K.~\cite{kar2007ttc}, and \blagojevic--M.--Ziegler~\cite{bmz2011ctv} proved further special cases and colorful extensions.

The following theorem and the subsequent Remark~\ref{remColorfulProjectiveTransversalTverberg} generalize all those results (which accumulate to the case $v=d-1$ below) and Theorem~\ref{projective-tver} (which is the special case $m=1$).

\begin{theorem}
\label{thmTransversalProjectiveTverberg}
Let $d > v\geq 0$, $m\geq 1$, $d>w:=m(d-v)-1$, $D:=(d-v)(d-w)$, and let $p$ be a prime.
Suppose that $p=2$ or $m=1$ or $d-w$ is even. Let $r_1=p^{\alpha_1},\ldots,r_m=p^{\alpha_m}$ be powers of~$p$.

Suppose that $X^1,\ldots,X^m$ are $m$ point sets in $\RP^d$ of size $|X^j| = (D+1)(r_j-1)+1$ and $V\subset\RP^d$ is a projective subspace of dimension~$v$.

Then there exists a projective subspace $W\subset\RP^d$ of dimension $w$ and partitions of each $X^j$ into $r_j$ pieces respectively,
\[
X^j = X_1^j\dotcup\ldots\dotcup X_{r_j}^j,
\]
such that any pair of hyperplanes $H_1\supseteq V$ and $H_2\supseteq W$ partitions $\RP^d$ into two closed parts so that each part has nonempty intersection with every~$X^j_i$.
\end{theorem}

\begin{remark}
It is natural to conjecture the same statement also for all positive integers $r_j$ and without the technical assumption of $p=2$ or $m=1$ or $d-w$ being even.
\end{remark}

\begin{remark}
The assumption that no $r$ points of any $X^i$ together with $V$ are contained in a hyperplane will guarantee that $W$ and $V$ intersect transversally.
\end{remark}

\begin{remark}
\label{remColorfulProjectiveTransversalTverberg}
In the case when all $r_j=p$, Theorem \ref{thmTransversalProjectiveTverberg} has a colorful extension analogously to the colorful Tverberg--\vrecica\ theorem of \cite{bmz2011ctv}:
If for every $j$ we color the point set $X^i$ so that every color appears at most $r_j-1$ times, then the partitions of the sets $X^j$ can be made rainbow colored in the sense that every part $X^j_i$ uses every color at most once.
\end{remark}

\begin{proof}
Let $n_j:=|X^j|=(D+1)(r_j-1)+1$ $(1\leq j\leq m)$.
As above we regard $V$ as a linear $(v+1)$-dimensional subspace of $\RR^{d+1}$, denoted by~$V'$.
Let $\widehat V\subseteq (\RR^{d+1})^*$ be its dual space of dimension $d-v$ (this is the same as going to the projective dual statement of the theorem).
We have to find a certain $(w+1)$-dimensional subspace $W'\subseteq\RR^{d+1}$.
We will instead find its dual space $\widehat W\subseteq (\RR^{d+1})^*$ of dimension $d-w$.

Note that the dimension of $\widehat V$ and $\widehat W$ will sum up to $2d-v-w$.
Let $B$ be the set of all $(2d-v-w)$-dimensional subspaces $U\subseteq (\RR^{d+1})^*$ that contain~$\widehat V$.
Sending $U$ to its quotient $U/\widehat V$ gives us a bijection $B\homeo G(v+1,d-w)$, which we use to topologize~$B$. The Grassmannian here may be considered oriented or not oriented; we defer this choice till the end of the proof.

Let us fix a subspace $U\in B$, $j\in\{1,\ldots,m\}$ and an $\eps>0$.
As in the proof of Theorem \ref{projective-tver}, we want to find $r_j$ quadratic forms $q_{p_i,\eps}$ on $U$ such that $p_1,\ldots,p_{r_j}\in \Delta^{n_j-1}$ have pairwise disjoint supports and such that the orthogonal complement of $\widehat V$ in $U$ with respect to any $q_{p_i,\eps}$ is the same $(d-w)$-subspace of~$U$.
As in the usual configuration space/test map scheme for topological Tverberg-type theorems, see for example \matousek~\cite{mat2003}, the $r_j$-tuples of coefficient vectors $(p_i)_{1\leq i\leq r_j}$ of those quadratic forms are the preimages of the thin diagonal of a test-map
\[
f_{U,j,\eps}: (\Delta^{n_j-1})^{r_j}_{\Delta(2)} \toG{\Sg_{r_j}} (S_U)^{r_j},
\]
where $S_U$ is the Schubert cell in $G(2d-w-v,d-w)$ that consists of all $(d-w)$-subspaces of $U$ that intersect $\widehat V$ trivially; and $(\Delta^{n_j-1})^{r_j}_{\Delta(2)}$ denotes the $r_j$-fold pairwise deleted product of $\Delta^{n_j-1}$, which is defined as the union of products $F_1\times\ldots\times F_{r_j}$ of pairwise disjoint faces of $\Delta^{n_j-1}$.
The test-map is naturally $\Sg_{r_j}$-equivariant, where $\Sg_{r_j}$ denotes the symmetric group on $r_j$ elements.

Elements $\widehat W\in S_U$ can be uniquely written as kernels of linear projections $U\to \widehat V$.
When choosing an inner product on $U$, those projections are in bijection with maps $U/\widehat V\to \widehat V$, which are elements in $(U/\widehat V)^*\otimes \widehat V$.
Thus,
\begin{equation}
\label{eqSchubertCellForU}
S_U\iso (U/\widehat V)^*\otimes \widehat V \homeo \RR^D
\end{equation}
where the first isomorphism is natural in $U\in B$ if the inner products on the $U$'s are consistently the restriction of a fixed inner product on $\RR^{d+1}$.
The maps $f_{U,j,\eps}$ depend continuously on $U\in B$.
Let $\gamma:E(\gamma)\to B$ denote the tautological vector bundle over $B\homeo G(v+1,d-w)$ of rank $d-w$, its fibers being $U/\widehat V$. The union of the Schubert cells $S_U$ forms in a natural way a vector bundle over~$B$ ($S_U$ being the fiber over $U\in B$), which by \eqref{eqSchubertCellForU} is isomorphic to $(\gamma^*)^{\oplus(d-v)}\iso \gamma^{\oplus(d-v)}$.

Hence the collection of maps $f_{U,j,\eps}$ $(U\in B)$ forms a $\Sg_{r_j}$-equivariant bundle map $f_{j,\eps}$ over~$B$,
\[
\xymatrix{
B\times (\Delta^{n_j-1})^{r_j}_{\Delta(2)}\ar[rr]^{f_{j,\eps}} \ar[dr]_{pr_1} && E(\gamma)^{\oplus r_j(d-v)}\supseteq \Delta \ar[dl]^{\gamma^{\oplus r_j(d-v)}}\\
& B &.
}
\]
Here, $\Delta\iso E(\gamma^{\oplus(d-v)})$ denotes the thin diagonal subbundle, and we are interested in the preimages $f_{j,\eps}^{-1}(\Delta)$.
In particular we have to show that
\begin{equation}
\label{eqIntersectionOfPreimagesOfDelta}
f_{1,\eps}^{-1}(\Delta)\cap\ldots\cap f_{m,\eps}^{-1}(\Delta)\neq\emptyset.
\end{equation}
This can be proved exactly as in \cite{kar2007ttc} or \cite{bmz2011ctv} using a parametrized Borsuk--Ulam type theorem.
The only topological assumption that we need is that the mod-$p$-Euler class of $\Delta$, $e(\Delta)=e(\gamma)^{d-v}\in H^D(B;\FF_p)$, satisfies
\begin{equation}
\label{eqEulerClassOfBCondition}
e(\Delta)^{m-1}\neq 0.
\end{equation}
We know that $e(\Delta)^{m-1}=e(\gamma)^{(v+1)-(d-w)}\neq 0$ if $p=2$ and we consider the nonoriented Grassmannian (see Hiller \cite{hil1980}) or if $p>2$, $\dim \gamma = d-w$ is even, and we consider the oriented Grassmannian (see~\cite[Lemma~8]{kar2007ttc}). The same is trivially true if $m=1$, since then $B$ is a single point $\pt$ and the ``Euler class'' is the generator of $H^0(\pt; \FF_p)$. Finally, \eqref{eqEulerClassOfBCondition} is holds true if $p=2$ or $d-w$ is even or $m=1$.

As in the previous proofs, letting $s$ go to zero finishes the proof.
\end{proof}

\section{A projective center transversal theorem} \label{secProjCenterTransversalThm}

As an immediate corollary of Theorem \ref{thmTransversalProjectiveTverberg} we obtain a projective generalization of \dolnikov's \cite{dol1987,dol1992,dol1994} and \zivaljevic's \cite{zivvre1990} center transversal theorem (which is the special case $v=d-1$):

\begin{corollary}[Projective center transversal theorem]
\label{projective-ctt}
Let $d > v\geq 0$, $m\geq 1$, $d>w:=m(d-v)-1$, and $D:=(d-v)(d-w)$.

Suppose that $\mu_1,\ldots,\mu_m$ are $m$ probability measures on $\RP^d$
and $V\subset\RP^d$ is a projective subspace of dimension~$v$.

Then there exists a projective subspace $W\subset\RP^d$ of dimension $w$ such that any pair of hyperplanes $H_1\supseteq V$ and $H_2\supseteq W$ partitions $\RP^d$ into two closed parts $P_1$ and $P_2$ so that for any $j\in\{1,\ldots,m\}$,
\[
\mu_j(P_1), \mu_j(P_2) \geq \frac{1}{D+1}.
\]
\end{corollary}

\begin{proof}
We approximate the measures $\mu_j$ by finite point sets $X^j$ (as in the proof of Tverberg and \vrecica~\cite{tv1993grthst} that their conjecture implies the center transversal theorem).
We apply Theorem \ref{thmTransversalProjectiveTverberg}, avoiding the parity assumption of $d-w$ being even by simply taking $p=2$ and $r_j$ equal to appropriate powers of two. Both parts $P_1$ and $P_2$ will contain at least one point from every $X^j_i$ and thus at least $r_j$ points from every $X^j$. Since $r_j/|X^j| \geq 1/(D+1)$ and $G(d+1,w+1)$ is compact, a limit argument finishes the proof.
\end{proof}

\begin{remark}
The case $m=1$ could be called a projective center point theorem in analogy with Rado's center point theorem \cite{rado1947}, which is the special case $m=1$ and $v=d-1$.
Similarly, the case $m=d$ and $v=w=d-1$ is the ham sandwich theorem; compare with \blagojevic--K.~\cite[Thm. 10]{bk2010} for an actual projective mass partition theorem.
\end{remark}

\section{Another projective Tverberg theorem} \label{secAnotherProjTver}

We could also ask what happens if we allow in Theorem~\ref{projective-tver}, for example, to choose $V$ arbitrarily.
A partial answer is given by the following:

\begin{theorem}
\label{projective-tver2}
Let $0\le v < d$, $r_1$ and $r_2$ be the powers of the same prime $p$, and $X^1$ and $X^2$ be finite point sets with $|X^j|=(D+1)(r_j-1)+1$, where $D=(v+1)(d-v)$.
We further assume that $D$ is even and $\binom{\lfloor\frac{d+1}{2}\rfloor}{\lfloor\frac{v+1}{2}\rfloor}$ is nonzero mod $p$.

Then there exist projective subspaces $V, W\subset \mathbb RP^d$ of dimensions $v$ and $d-v-1$ and partitions of every $X^j$ into $r_j$ pieces
\[
X^j = X_1^j\dotcup\ldots\dotcup X_{r_j}^j,
\]
such that any pair of hyperplanes $H_1\supseteq V$ and $H_2\supseteq W$ partitions $\mathbb RP^d$ into two parts so that each part has nonempty intersection with every $X^j_i$.
\end{theorem}

\begin{remark}
It was already known to Kummer that $\binom{n}{k}$ is nonzero mod $p$ if and only if for all $i\geq 0$ the $i$'th digit of $n$ in the $p$-ary representation is greater or equal to the $i$'th digit of $k$.
\end{remark}

\begin{remark}
As a corollary we obtain a corresponding center point theorem for partitioning two measures under the assumption that $D$ is even. The proof is analogous to the one of Corollary~\ref{projective-ctt}.
\end{remark}

\begin{proof}
We follow the proof of Theorem~\ref{projective-tver} for $X^j$ and describe the set of appropriate $\widehat W_j$ for every $\widehat V$ as a preimage.
We have to find some $\widehat V$ such that $\widehat W_1$ and $\widehat W_2$ are equal. Their difference is an element of $(\mathbb R^{d+1}/\widehat V) \otimes \widehat V$. So (the formal argument is similar to the proof of Theorem~\ref{thmTransversalProjectiveTverberg}, using a parameterized testmap) we observe that the spaces $(\mathbb R^{d+1}/\widehat V) \otimes \widehat V$ form a vector bundle over the Grassmannian $G(d+1, v+1)$ of all possible $V$ and this bundle is isomorphic to $\gamma\otimes \gamma^\perp$, where $\gamma$ is the tautological bundle over the Grassmannian and $\gamma^\perp$ its orthogonal complement bundle.

It remains to check that the Euler class $e(\gamma\otimes\gamma^\perp)\in H^D(G(d+1,v+1);\mathcal F_p)$ is nonzero (the coefficients are $\FF_p$ with possible twist corresponding to the orientation character of $\gamma\otimes\gamma^\perp$).
An open neighborhood (and thus the tangent space) of $G(d+1,v+1)$ at a point $\widehat V$ can be identified with
$$
\widehat V^*\otimes \widehat V^\perp = \Hom (\widehat V, \widehat V^\perp),
$$
because adding such maps to the inclusion $\widehat V\incl\RR^{d+1}$ gives linear maps $\widehat V\to \mathbb R^{d+1}$ with all possible images intersecting $\widehat V^\perp$ trivially.
Thus the images of such maps constitute the Schubert cell neighborhood of $\widehat V$.
Since the identifications are natural and $\gamma^*\iso\gamma$, we have:
$$
\gamma\otimes\gamma^\perp \iso TG(d+1, v+1).
$$
Thus the Euler class $e(\gamma\otimes\gamma^\perp)$ is Poincar\'e dual to the element of $H_0(G(d+1, v+1); \FF_p)$ equal to the Euler characteristic of $G(d+1,v+1)$ mod $p$; the cohomology twist is eliminated because it has to be tensor multiplied by the orientation sheaf of $TG(d+1, v+1)$. The Euler characteristic of the Grassmannian is equal to the value of the $q$-binomial coefficient $\binom{d+1}{v+1}_q$ at $q=-1$. This value can be calculated by the inductive formula:
$$
\binom{n}{k}_{-1} = \binom{n-1}{k-1}_{-1} + (-1)^k \binom{n-1}{k}_{-1},\ \ \ \binom{n}{0}_{-1}=\binom{n}{n}_{-1}=1,
$$
and it finally turns out to be zero if $D$ is odd and equal to
$$
\binom{\lfloor\frac{d+1}{2}\rfloor}{\lfloor\frac{v+1}{2}\rfloor}
$$
if $D$ is even.
\end{proof}

\section{A corresponding ``transversal'' generalization} \label{secTranversalGeneralizationOfTverberg}

In this section we extend Theorem \ref{projective-tver2} under some algebraic assumptions to the case of more than two masses in the same way as Theorem \ref{thmTransversalProjectiveTverberg} extends Theorem \ref{projective-tver}.
The pay-off for establishing this is that the dimension of $V$ and $W$ must increase.

Let $0\leq v,w < d$, $m\geq 1$, $r_j=2^{\alpha_j}$ $(1\leq j\leq m)$, $\widehat v := d-v$, $\widehat w := d-w$, $D:= \widehat v\widehat w$.
Suppose that $X^1,\ldots,X^m$ are $m$ point sets in $\RP^d$ of size $|X^j|=(D+1)(r_j-1)+1$.
We will find a condition under which there exist projective subspaces $V,W\subset\RP^d$ of dimensions $v$ and $w$ and partitions of every $X^j$ into $r_j$ pieces
\[
X^j = X_1^j\dotcup\ldots\dotcup X_{r_j}^j,
\]
such that any pair of hyperplanes $H_1\supseteq V$ and $H_2\supseteq W$ partitions $\RP^d$ into two parts so that each part has nonempty intersection with every $X^j_i$.

Similarly to the above sections, the topological setup for this is the following.
Let $B:=F(\widehat v, \widehat v+\widehat w, d+1)$ be the partial flag manifold of all pairs $(\widehat V, U)$ of linear subspaces of $\RR^{d+1}$ such that $\dim(\widehat V)=\widehat v$, $\dim{U}=\widehat v+\widehat w$, and $\widehat V\subset U$.
Let $\eta$ be the rank $D$ vector bundle over $B$ whose fiber over $(\widehat V,U)$ is $\widehat V\otimes (U/\widehat V)$.
Let $w_i(\eta)\in H^i(B;\FF_2)$ denote its Stiefel--Whitney classes.
The condition we look for is
\begin{equation}
\label{eqWofEta}
w_D(\eta)^{m-1}\neq 0.
\end{equation}

We can put this condition into purely algebraic terms as follows using a particular instance of the splitting principle.
Let $F:=\textnormal{Flag}(\RR^{d+1})=F(1,2,\ldots,d+1)$ be the complete real flag manifold in $\RR^{d+1}$.
There is a canonical projection $p:F\to B$.
Let $\ell_1,\ldots,\ell_{d+1}$ be the tautological line bundles over $F$.
Their first Stiefel--Whitney classes $e_i:=w_1(\ell_i)$ generate the ring $H^*(F;\FF_2)$.
More exactly,
\begin{equation}
\label{eqCohomologyOfFlagManifold}
H^*(F;\FF_2)\ \iso\ \FF_2[e_1,\ldots,e_{d+1}]/\left(\prod (1+e_i) = 1\right).
\end{equation}
The relation means geometrically that the sum $\ell_1\oplus\ldots\oplus\ell_{d+1}$ is trivial and algebraically that all elementary symmetric polynomials in $e_1,\ldots,e_{d+1}$ vanish.
The reader may find the latter fact and other standard facts about vector bundles and their characteristic classes in~\cite{mishch1998}.

Now,
\begin{align*}
w(p^*\eta) &= w\big((\ell_1\oplus\ldots\oplus\ell_{\widehat v})\otimes(\ell_{\widehat v+1}\oplus\ldots\oplus\ell_{\widehat v+\widehat w})\big) \\
&= {\prod_{i,j}}^* w(\ell_i\otimes \ell_j)\ =\ {\prod_{i,j}}^* \big(1+w_1(\ell_i)+w_1(\ell_j)\big),
\end{align*}
where $\prod_{i,j}^*$ denotes that the product runs over all pairs $(i,j)$ with $1\leq i\leq \widehat v$ and $\widehat v+1\leq j\leq\widehat v+\widehat w$.
In particular,
\begin{equation}
\label{eqSWofEta}
w_D(p^*\eta) = {\prod_{i,j}}^* (e_i+e_j).
\end{equation}
Thus $\eqref{eqWofEta}$ is equivalent to $\prod_{i,j}^* (e_i+e_j)^{m-1}\neq 0$ in $H^*(F;\FF_2)$.
This proves:

\begin{theorem}
Let $0\leq v,w < d$, $m\geq 1$, $r_j=2^{\alpha_j}$ $(1\leq j\leq m)$, $D:= (d-v)(d-w)$, such that the product \eqref{eqSWofEta} is non-zero in \eqref{eqCohomologyOfFlagManifold}.

Suppose that $X^1,\ldots,X^m$ are $m$ point sets in $\RP^d$ of size $|X^j|=(D+1)(r_j-1)+1$.

Then there exist projective subspaces $V,W\subset\RP^d$ of dimensions $v$ and $w$ and partitions of every $X^j$ into $r_j$ pieces
\[
X^j = X_1^j\dotcup\ldots\dotcup X_{r_j}^j,
\]
such that any pair of hyperplanes $H_1\supseteq V$ and $H_2\supseteq W$ partitions $\RP^d$ into two parts so that each part has nonempty intersection with every $X^j_i$.
\end{theorem}


\end{document}